\documentclass[a4paper,11pt]{amsart}

\usepackage{amsfonts,amssymb}
\usepackage{textcomp}

\theoremstyle{plain}

\newtheorem{thm}{Theorem}[section]
\newtheorem{pro}[thm]{Proposition}
\newtheorem{lem}[thm]{Lemma}
\newtheorem{cor}[thm]{Corollary}

\theoremstyle{definition}

\newtheorem{rmk}[thm]{Remark}

\newcommand{\F}{\mathbb{F}}
\newcommand{\Z}{\mathbb{Z}}

\newcommand{\N}{\mathbb{N}}

\DeclareMathOperator{\Aut}{Aut}

\DeclareMathOperator{\Dih}{Dih}

\begin{document}

\title[Finiteness conditions on normalizers or centralizers]
{Some finiteness conditions on normalizers\\ or centralizers in groups}

\author[G.A. Fern\'andez-Alcober]{Gustavo A. Fern\'andez-Alcober}
\address{Matematika Saila\\ Euskal Herriko Unibertsitatea UPV/EHU\\
48080 Bilbao, Spain. {\it E-mail address}: {\tt gustavo.fernandez@ehu.eus}}

\author[L. Legarreta]{Leire Legarreta}
\address{Matematika Saila\\ Euskal Herriko Unibertsitatea UPV/EHU\\
48080 Bilbao, Spain. {\it E-mail address}: {\tt leire.legarreta@ehu.eus}}

\author[A. Tortora]{Antonio Tortora}
\address{Dipartimento di Matematica\\ Universit\`a di Salerno\\
Via Giovanni Paolo II, 132\\ 84084 Fisciano (SA)\\ Italy. {\it E-mail address}: {\tt antortora@unisa.it}}

\author[M. Tota]{Maria Tota}
\address{Dipartimento di Matematica\\ Universit\`a di Salerno\\
Via Giovanni Paolo II, 132\\ 84084 Fisciano (SA)\\ Italy. {\it E-mail address}: {\tt mtota@unisa.it}}

\thanks{The first two authors are supported by the Spanish Government, grants
MTM2011-28229-C02-02 and MTM2014-53810-C2-2-P, and by the Basque Government, grants IT753-13 and IT974-16.
The last two authors would like to thank the Department of Mathematics at the University of the Basque Country for its excellent hospitality while part of this paper was being written; they also wish to thank G.N.S.A.G.A. (INdAM) for financial support.}

\keywords{Normalizers, centralizers, locally finite groups, locally nilpotent groups\vspace{3pt}}
\subjclass[2010]{20E99}

\begin{abstract}
We consider the following two finiteness conditions on normalizers and centralizers in a group $G$: (i) $|N_G(H):H|<\infty$ for every $H\ntriangleleft G$, and (ii) $|C_G(x):\langle x \rangle|<\infty$ for every $\langle x \rangle \ntriangleleft G$.
We show that (i) and (ii) are equivalent in the classes of locally finite groups and locally nilpotent groups.
In both cases, the groups satisfying these conditions are a special kind of cyclic extensions of Dedekind groups.
We also study a variation of (i) and (ii), where the requirement of finiteness is replaced with a bound.
In this setting, we extend our analysis to the classes of periodic locally graded groups and non-periodic groups.
While the two conditions are still equivalent in the former case, in the latter the condition about normalizers is stronger than that about centralizers.
\end{abstract}

\maketitle

\section{Introduction}

Motivated by results of Q.\ Zhang and Gao \cite{zha-gao} and X.\ Zhang and Guo \cite{zha-guo}, we considered in \cite{FLTT2} the finite $p$-groups $G$, where $p$ is a prime, satisfying one of the following three conditions for a given $n$:
that $|N_G(H):H|\le n$ for every non-normal subgroup $H$ of $G$, or that either
$|N_G(\langle x \rangle):\langle x \rangle|\le n$ or $|C_G(x):\langle x \rangle|\le n$ for every non-normal cyclic subgroup $\langle x \rangle$ of $G$.
In \cite{FLTT} we also dealt with the last of these conditions with non-central elements instead of elements generating a non-normal subgroup, in the case of general finite groups.

Later, in \cite{FLTT3} and \cite{FLTT4}, we extended our earlier work to the realm of infinite groups. Following \cite{FLTT3}, we say that a group $G$ is
an FCI-group (FCI for `finite centralizer index') provided that
\begin{equation}
\label{FCI}
|C_G(x):\langle x \rangle| < \infty
\quad
\text{for every $\langle x \rangle \ntriangleleft G$,}
\end{equation}
and if there exists $n$ such that
\begin{equation}
\label{BCI}
|C_G(x):\langle x \rangle| \le n
\quad
\text{for every $\langle x \rangle \ntriangleleft G$,}
\end{equation}
then we call $G$ a BCI-group (BCI for `bounded centralizer index').
Similarly, we can require that
\begin{equation}
\label{FNI for all}
|N_G(H):H| < \infty
\quad
\text{for every $H \ntriangleleft G$,}
\end{equation}
and then we say that $G$ is an FNI-group (FNI for `finite normalizer index'), or the seemingly weaker condition that this only holds for non-normal cyclic subgroups of $G$:
\begin{equation}
\label{FNI}
|N_G(\langle x \rangle):\langle x \rangle| < \infty
\quad
\text{for every $\langle x \rangle \ntriangleleft G$.}
\end{equation}
We can also ask for the existence of a uniform bound for the indices in (\ref{FNI for all}) and
(\ref{FNI}), thus getting the conditions that, for some positive integer $n$,
\begin{equation}
\label{BNI for all}
|N_G(H):H| \le n
\quad
\text{for every $H \ntriangleleft G$,}
\end{equation}
in which case we speak of BNI-groups (BNI for `bounded normalizer index'), or that
\begin{equation}
\label{BNI}
|N_G(\langle x \rangle):\langle x \rangle| \le n
\quad
\text{for every $\langle x \rangle \ntriangleleft G$.}
\end{equation}
Of course, there are obvious connections between these conditions.
Observe that (\ref{FNI for all}) implies (\ref{FNI}), which is equivalent to (\ref{FCI});
on the other hand, (\ref{BNI for all}) implies (\ref{BNI}), which in turn implies (\ref{BCI}).

All the conditions above hold in finite groups, so they can be considered as finiteness conditions, and they are clearly hereditary for subgroups. Also, (\ref{FNI for all}) and (\ref{BNI for all}) can be easily checked to be hereditary for quotients. However, the rest of the conditions do not usually transfer to quotients. For example, consider the generalized dihedral group $G=\Dih(A)$, where $A$ is torsion-free abelian of infinite $0$-rank.
Then $G$ satisfies conditions (\ref{FCI}), (\ref{BCI}), (\ref{FNI}), and (\ref{BNI}), but if $N$ is the subgroup consisting of all fourth powers of elements of $A$, then
$G/N\cong\Dih(A/N)$ does not satisfy any of them (see \cite[Section 2]{FLTT3} for more details).

In \cite{FLTT3} we considered infinite locally finite FCI-groups, whereas in \cite{FLTT4} we dealt with infinite locally nilpotent FCI-groups, and also with non-periodic BCI-groups.
In every case, the groups in question can be characterized as a special type of cyclic extensions of certain Dedekind groups.
Furthermore, in \cite{FLTT3}, we proved that every locally graded periodic BCI-group is locally finite, where the restriction to locally graded groups is imposed to avoid the presence of Tarski monster groups (i.e.\ infinite simple $p$-groups, for $p$ a prime, all of whose proper non-trivial subgroups are of order $p$). Recall that a group is locally graded if every non-trivial finitely generated subgroup has a nontrivial finite image: for example, locally (soluble-by-finite) groups, as well as residually finite groups, are locally graded.

The goal of this paper is to study the conditions about normalizers in the same cases that we have considered for centralizers.
For locally finite groups, we will show that the conditions for normalizers and centralizers are all equivalent (Theorem \ref{normalizers locally finite}).
As a consequence, we obtain the equivalence between the BNI- and BCI-conditions for periodic locally graded groups (Corollary \ref{locally graded}).
In the locally nilpotent case, it suffices to deal with non-periodic groups. In this situation, we know that BCI-groups are abelian \cite[Corollary 4.4]{FLTT4}, and so the same holds for BNI-groups.
Hence we only consider the case of FNI-groups, and again we obtain that they are the same as FCI-groups (Theorem \ref{normalizers locally nilpotent}).
Finally, contrary to the previous cases, non-periodic BNI-groups are not the same as non-periodic BCI-groups (Theorem \ref{normalizers non-periodic}).
On the other hand, condition (\ref{BNI}) turns out to be equivalent to being a BCI-group for the classes of periodic locally graded groups and non-periodic groups.
However, it is not clear to us whether this holds in all generality.

\vspace{10pt}

\noindent
Most of our notation is standard and can be found in \cite{Ro}. In particular, if $G$ is a periodic group, we write $\pi(G)$ for the set of prime divisors of the orders of the elements of $G$; if $G$ is also nilpotent and $\varphi$ is an automorphism of $G$, we denote by $G_p$ the unique Sylow $p$-subgroup of $G$, and by $\varphi_p$ the restriction of $\varphi$ to $G_p$. Finally, we write $\Z_p$ for the ring of $p$-adic integers and we denote by $\Z_p^{\times}$ its group of units.

\section{Locally finite groups}

In this section we will prove the equivalence, in the realm of locally finite groups, between the conditions for normalizers and centralizers described in the introduction. A crucial role will be played by the determination of infinite locally finite FCI-groups \cite{FLTT3}, which is in turn related to the analysis of the centralizer of a power automorphism.

Recall that an automorphism of a group $G$ is said to be a
\emph{power automorphism\/} if it sends every element $x\in G$ to a power of $x$. Power automorphisms form an abelian subgroup of $\Aut G$.
According to a result of Robinson \cite[Lemma 4.1.2]{Ro3} (see also \cite[Section 2]{FLTT4}), every power automorphism $\varphi$ of an abelian
$p$-group $A$ can be written in the form $a\mapsto a^t$ for all $a\in A$ and for some
$t\in\Z_p^{\times}$.
We then say that $t$ is an \emph{exponent\/} for $\varphi$; following \cite{FLTT4}, we  write $\exp \varphi=t$.
Observe, however, that $\exp\varphi$ is only uniquely determined if the group $A$ is of infinite exponent.

We require the following result, which is Lemma 3.2, part (i), of \cite{FLTT3}.

\begin{lem}
\label{cent of auto}
Let $A$ be an abelian $p$-group, where $p$ is a prime, and let $\varphi$ be a power automorphism of $A$ of finite order $m>1$.
If $m$ divides $p-1$, then $C_A(\varphi^k)=1$ for every $k\in\{1,\ldots,m-1\}$.
\end{lem}

For the reader's convenience, we also provide the characterization of infinite locally finite FCI-groups (see \cite[Theorem 4.4 and Remark 4.7]{FLTT3}).

\begin{thm}
\label{determination locally finite FCI-groups}
A non-Dedekind group $G$ is an infinite locally finite FCI-group if and only $G=\langle g,D \rangle$, where $D$ is an infinite periodic Dedekind group with $D_2$ of finite rank, and $g$ acts on $D$ as a power automorphism $\varphi$ of order $m>1$ such that
$|G:D|=m$ and
\begin{equation*}
\label{condition for locally finite}
\prod_{p\in\pi_0\cup\pi_1} \, |D_p| < \infty,
\end{equation*}
where
$$
\pi_0 = \{ p\in \pi(D) \mid o(\varphi_p)<m \},
$$
and
$$
\pi_1 = \{ p\in \pi(D) \mid \text{$o(\varphi_p)=m$, $p>2$ and
$p\not\equiv 1\hspace{-10pt}\pmod m$} \}.
$$
Furthermore, if $D_2$ is infinite, then $D_2$ is abelian, $\varphi_2$ is the inversion automorphism, and $m=2$.
\end{thm}

Recall (see, for instance, \cite[Section 5.1]{LR}) that the rank of an abelian $p$-group is the dimension, as a vector space over
the field with $p$ elements, of the subgroup formed by the elements of order at most $p$. If $A$ is an abelian group, the $p$-rank of $A$ is defined as the rank of $A_p$, and the torsion-free rank, or $0$-rank, of $A$ is the cardinality of a maximal independent subset of elements of $A$ of infinite order. These are denoted by $r_p(A)$ and $r_0(A)$, respectively. The (Pr\"ufer) rank of $A$ is
\[
r(A) = r_0(A) + \max_p \, r_p(A),
\]
where $p$ runs over all prime numbers. More generally, a group $G$ is said to have finite rank $r=r(G)$ if every finitely generated subgroup of $G$ can be generated by $r$ elements and $r$ is the least such integer.

We know from \cite[Corollary 4.5]{FLTT3} that every locally finite FCI-group is a BCI-group.
Next we prove that actually every quotient of a locally finite FCI-group is a BCI-group, and furthermore with a uniform bound.
Note that this can also be seen as a generalization, in the class of locally finite groups, of Proposition 2.2 of \cite{FLTT3}, according to which every quotient of an FCI-group by a finite normal subgroup is again an FCI-group.

\begin{pro}
\label{pro}
\label{G/N FCI locally finite}
Let $G$ be a locally finite FCI-group.
Then there exists $n\geq 1$ depending only on $G$ such that, whenever $N\lhd G$ and
$\langle xN \rangle \ntriangleleft G/N$, we have $|C_{G/N}(xN)| \le n$.
In particular, $G/N$ is a BCI-group for every normal subgroup $N$ of $G$.
\end{pro}

\begin{proof}
We may assume that $G$ is infinite and non-Dedekind.
Then $G=\langle g,D \rangle$, as in
Theorem \ref{determination locally finite FCI-groups}, and we use the notation therein throughout the proof.
If $D_2$ is finite, let $A$ be the Hall $2'$-subgroup of $D$ and let $B=D_2$; otherwise, put $A=D$ and $B=1$.
In any case we have $D=A\times B$ with $A\le Z(D)$ and $B$ finite.
We write $\overline G$ for $G/N$, and we denote by $\overline\varphi$ and
$\overline\varphi_p$ the automorphisms induced by $\varphi$ on $\overline A$ and
$\overline A_p$, respectively.

Since every element of $D$ generates a normal subgroup of $G$, in order to prove the theorem we consider an arbitrary element $x\in G\smallsetminus D$, in which case we show that
\begin{equation}
\label{general bound}
|C_{\overline G}(\overline x)|
\le
m M \prod_{p\in (\pi_0\cup\pi_1)\smallsetminus \{2\}} \, |D_p|,
\end{equation}
where $M=|D_2|$ or $2^{r(D_2)}$, according as $D_2$ is finite or infinite.
Let us write $x=g^kd$, with $k\in\{1,\ldots,m-1\}$ and $d\in D$.
Then
\[
|C_{\overline G}(\overline x)|
\le
|\overline G:\overline A| \, |C_{\overline A}(\overline x)|
\le
m \, |B| \, |C_{\overline A}(\overline\varphi^k)|
=
m \, |B| \, \prod_{p\in \pi(\overline A)} \, |C_{\overline A_p}(\overline\varphi^k)|.
\]
Now if $p\in(\pi_0\cup\pi_1)\smallsetminus\{2\}$ then
$|C_{\overline A_p}(\overline\varphi^k)|\le |D_p| < \infty$.
On the other hand, if $p=2$ and $\overline A_2$ is non-trivial then $D_2$ is infinite abelian, $\varphi_2$ is the inversion automorphism and $m=2$.
Hence $|C_{\overline A_2}(\overline\varphi)|\le 2^{r(D_2)}$.

Thus (\ref{general bound}) will be proved once we see that
$C_{\overline D_p}(\overline\varphi^k)$ is trivial for $p\not\in \pi_0\cup\pi_1\cup\{2\}$.
Assume that $\overline D_p\ne \overline 1$.
Since $o(\varphi_p)=m$ divides $p-1$, if we see that also
$o(\overline\varphi_p)=m$, then the result follows from Lemma \ref{cent of auto}.
Assume by way of contradiction that the order of $\overline\varphi_p$, say $k$, is smaller than $m$.
If $t=\exp\varphi_p^k$ then $t\not\equiv 1\pmod p$; otherwise $\varphi_p^k$ fixes all elements of order $p$ in $D_p$, in contrast with Lemma \ref{cent of auto}.
Then for every $d\in D_p$ we have $\varphi_p^k(d)=dn$ for some $n\in N$, and on the other hand, $\varphi_p^k(d)=d^{\,t}$.
Consequently, $n=d^{\,t-1}$ is a generator of $\langle d \rangle$ and $d\in N$.
Hence $\overline D_p=\overline 1$, which is a contradiction.
\end{proof}

We are now ready to prove the main result of this section.

\begin{thm}
\label{normalizers locally finite}
Let $G$ be a locally finite group.
Then all the following conditions are equivalent:
\begin{enumerate}
\item
$G$ is an FNI-group.
\item
$G$ is an FCI-group.
\item
$G$ is a BNI-group.
\item
$G$ is a BCI-group.
\item
$G$ satisfies condition (\ref{FNI}).
\item
$G$ satisfies condition (\ref{BNI}).
\end{enumerate}
\end{thm}

\begin{proof}
Since locally finite FCI-groups are known to be BCI-groups, it readily follows that we only need to prove that every locally finite BCI-group $G$ is a BNI-group.
We may assume that $G$ is infinite and non-Dedekind.
Then Theorem \ref{determination locally finite FCI-groups} applies, and we write
$G=\langle g,D \rangle$ and $m=|G:D|>1$, as in that theorem.
Let $H$ be a non-normal subgroup of $G$, and observe that $H$ is not contained in $D$, since every subgroup of $D$ is normal in $G$.
We have
\[
|N_G(H):H| = |N_G(H)D:HD| |N_D(H):H\cap D| \le m |N_D(H):H\cap D|,
\]
and consequently it suffices to prove that $|N_D(H):H\cap D|$ is finite and can be bounded as $H$ runs over all non-normal subgroups of $G$.
Since $G/D$ is cyclic, we can write $H=\langle h,H\cap D \rangle$ for some
$h\in H\smallsetminus D$.

Now observe that an element $x\in G$ lies in $N_G(H)$ if and only if $[x,h]\in H$, since
$H\cap D\lhd G$.
It follows that
\[
N_D(H) = \{ d\in D \mid [d,h]\in H\cap D \},
\]
and then in the quotient group $\overline G=G/(H\cap D)$ we have
\[
\overline {N_D(H)} = C_{\overline D} (\overline h).
\]
Thus $|N_D(H):H\cap D|=|C_{\overline D} (\overline h)|$.
Now, since $G$ is a locally finite FCI-group and
$\langle \overline h \rangle \ntriangleleft \overline G$, it follows from Proposition \ref{G/N FCI locally finite} that $|C_{\overline G}(\overline h)|\le n$, for some $n$ which depends only on the group $G$. This proves the result.
\end{proof}

Since periodic locally graded BCI-groups are locally finite \cite[Theorem 5.3]{FLTT3}, we have the following corollary of Theorem \ref{normalizers locally finite}.

\begin{cor}
\label{locally graded}
Let $G$ be a periodic locally graded group.
Then the following conditions are equivalent:
\begin{enumerate}
\item
$G$ is a BNI-group.
\item
$G$ is a BCI-group.
\item
$G$ satisfies condition (\ref{BNI}).
\end{enumerate}
\end{cor}

\section{Non-periodic groups}

We start this section by considering infinite locally nilpotent FNI- and BNI-groups.
By Theorem \ref{normalizers locally finite}, we only need to deal with non-periodic groups.
On the other hand, we know from \cite[Corollary 4.4]{FLTT4} that in this situation BCI-groups are abelian, and so the same holds for BNI-groups.
Hence it suffices to study the case of FNI-groups.
To this purpose, we recall the characterization of non-periodic locally nilpotent FCI-groups which we obtained in \cite[Theorem 3.7]{FLTT4}.

\begin{thm}
\label{non-periodic locally nilpotent}
A non-periodic group $G$ is a locally nilpotent FCI-group if and only if either $G$ is abelian or $G=\langle g\rangle\ltimes D$, where $g$ is of infinite order, $D$ is a Dedekind group which is the product of finitely many $p$-groups of finite rank, and $g$ acts on $D$ as a power automorphism $\varphi$ satisfying the following conditions:
\begin{enumerate}
\item
If $D_2$ is non-abelian, $\varphi_2$ is the identity automorphism.
\item
If $p>2$ then $\exp \varphi_p\equiv 1 \pmod p$, and $\exp \varphi_p\neq 1$ if $D_p$ is infinite.
\item
If $p=2$ and $D_2$ is infinite, then $\exp \varphi_2\neq 1, -1$.
\end{enumerate}
\end{thm}

\begin{rmk}
In the previous statement, an expression like $\exp\varphi\ne t$ does not mean that $t$ is not a exponent for $\varphi$, but rather that we can choose an exponent for $\varphi$ which is different from $t$.
\end{rmk}

Next we give a generalization of Proposition 2.2 of \cite{FLTT3} in the class of locally nilpotent groups, as we did in Proposition \ref{G/N FCI locally finite} for locally finite groups.

\begin{pro}
\label{G/N FCI locally nilpotent}
Let $G$ be a locally nilpotent FCI-group, and let $N$ be a normal subgroup of $G$.
Then $G/N$ is an FCI-group.
\end{pro}

\begin{proof}
We may assume that $G$ is non-periodic and non-abelian.
Then we can write $G=\langle g \rangle \ltimes D$, as in
Theorem \ref{non-periodic locally nilpotent}, and we use the notation in the
statement of that theorem.
Also, we use the bar notation in the quotient $G/N$.

Suppose first that $N$ is not contained in $D$.
We claim that $\overline G$ is finite in this case.
The quotient $G/D$ is infinite cyclic, and so it is clear that $\overline G/\overline D$ is finite.
Hence it suffices to see that $\overline D_p$ is finite for every $p\in\pi(D)$, since $\pi(D)$ is finite.

Let us fix $p\in\pi(D)$ and assume by way of contradiction that $\overline D_p$ is infinite.
Since $D_p$ is of finite rank, it follows that $\overline D_p$, and so also $D_p$, is of infinite exponent.
As a consequence, $t_p=\exp\varphi_p$ is a uniquely determined element in $\Z_p$.
By (ii) and (iii) of Theorem \ref{non-periodic locally nilpotent}, we have
$t_p\equiv 1\pmod p$ and $t_p\ne 1$ if $p>2$ and $t_2\ne 1,-1$ if $p=2$.
In any case, $t_p$ is an element of infinite order in the multiplicative group
$\Z_p^{\times}$.
Now consider an element $x=g^kd\in N$ with $k\ge 1$ and $d\in D$.
Since $D$ is a Dedekind group and $D_p$ is abelian if $p=2$ by (i) and (iii) of
Theorem \ref{non-periodic locally nilpotent}, it follows that $D_p\le Z(D)$ in any case. Thus $d$ acts trivially on $D_p$, and consequently
\[
\overline 1 = \overline D_p^{\overline x-\overline 1}
= \overline D_p^{\varphi_p^k-1} = \overline D_p^{t_p^k-1}.
\]
Since $\overline D_p$ is of infinite exponent, it follows that $t_p^k-1$ is divisible by arbitrarily high powers of $p$, and consequently $t_p^k-1=0$.
Thus $t_p$ is of finite order in $\Z_p^{\times}$, and this contradiction proves the claim.

Now suppose that $N$ is contained in $D$.
Put $N=\prod_{p\in\pi(D)} \, N_p$, where $N_p\le D_p$ is the Sylow $p$-subgroup of $N$.
We have $\overline G=\langle \overline g \rangle \ltimes \overline D$, with $\overline g$
of infinite order and $\overline D_p\cong D_p/N_p$.
Now if $\overline \varphi_p$ is the automorphism of $\overline D_p$ induced by conjugation by $\overline g$, then we can choose
$\exp \overline\varphi_p$ equal to $\exp \varphi_p$ for every $p\in\pi(D)$.
As a consequence, $\overline G$ fulfills all conditions which are required in
Theorem \ref{non-periodic locally nilpotent} to conclude that it is an FCI-group.
\end{proof}

Now an argument which is similar to that of Theorem \ref{normalizers locally finite}, with Proposition \ref{G/N FCI locally nilpotent} playing the role of Proposition \ref{G/N FCI locally finite}, yields that non-periodic locally nilpotent FNI-groups are the same as FCI-groups.
Thus we have the following result.

\begin{thm}\label{normalizers locally nilpotent}
Let $G$ be a non-periodic locally nilpotent group.
Then the following conditions are equivalent:
\begin{enumerate}
\item
$G$ is an FNI-group.
\item
$G$ is an FCI-group.
\item
$G$ satisfies condition (\ref{FNI}).
\end{enumerate}
\end{thm}

We finish by considering non-periodic groups in general which satisfy one of the conditions for normalizers with bounds, i.e.\ that are either BNI-groups or satisfy (\ref{BNI}).

\begin{thm}\label{normalizers non-periodic}
Let $G$ be a non-periodic group.
Then the following hold:
\begin{enumerate}
\item
$G$ is a BNI-group if and only if either $G$ is abelian or $G=\langle g, A \rangle$, where $A$ is a non-periodic abelian group of finite $0$-rank and finite $2$-rank, and $g$ is an element of order at most $4$ such that $g^2\in A$ and $a^g=a^{-1}$ for all $a\in A$.
\item
$G$ satisfies (\ref{BNI}) if and only if $G$ is a BCI-group.
\end{enumerate}
\end{thm}

\begin{proof}
(i)
Let us first assume that $G$ is a BNI-group.
Then $G$ is also a BCI-group, and then, by Theorem 4.3 of \cite{FLTT4}, $G$ is either abelian or of the form $G=\langle g, A \rangle$, where $A$ is a non-periodic abelian group of finite $2$-rank, and
$g$ is an element of order at most $4$ such that $g^2\in A$ and $a^g=a^{-1}$ for all
$a\in A$.
Thus we only need to prove that, in the latter case, $A$ is also of finite $0$-rank.

Arguing by contradiction, let us consider an infinite sequence $\{b_n\}_{n\in\N}$ of linearly independent elements of infinite order inside $A$.
Then $B=\langle b_n \mid n\in\N \rangle$ is torsion-free.
Let us put $H=\langle g,B \rangle$ and $N=\langle g^2 \rangle$.
Since $N$ is contained in $A$, then $N$ is normal in $G$ and the quotient $H/N$ is isomorphic to the generalized dihedral group $\Dih(B)$ (i.e. the semidirect product of $B$ with a cyclic group of order $2$ acting on $B$ by inversion).
Now, since $B$ is torsion-free of infinite $0$-rank, there are quotients of $\Dih(B)$ which are not BCI-groups: for example, if $K$ is the subgroup consisting of all fourth powers of elements of $B$, then $\Dih(B)/K\cong \Dih(B/K)$ is not a BCI-group, because $B/K$ is of infinite 2-rank (see \cite[Section 2]{FLTT3}).
This is a contradiction, since every section of $G$ is a BNI-group (recall that the BNI-condition is hereditary for subgroups and quotients), and consequently also a BCI-group.

Now we prove the converse. To this purpose, we consider a group
$G=\langle g, A \rangle$ as in the statement of the theorem, and see that $G$ is a
BNI-group.
More precisely, if $H$ is a non-normal subgroup of $G$, we are going to show that
$|N_G(H):H|\le 2^{r_0+r_2}$, where $r_0$ and $r_2$ are the $0$-rank and the $2$-rank of $A$, respectively.

Since all subgroups of $A$ are normal in $G$, it follows that $H$ is not contained in $A$.
Thus $H=\langle h \rangle B$, where $h\in H\smallsetminus A$ and $B=H\cap A$.
Since $|G:A|=2$, we have $G=\langle h \rangle A$.
Then $N_G(H)=\langle h \rangle N_A(H)$ and  $|N_G(H):H| = |N_A(H):B|$.
Observe that, for a given $a\in A$, we have $a\in N_A(H)$ if and only if $[h,a]\in H$, i.e.\
if and only if $a^2\in B$.
As a consequence, $N_A(H)/B$ is an elementary abelian $2$-group.

Thus in order to prove that $|N_G(H):H|\le 2^{r_0+r_2}$, it suffices to show that the following holds: for any choice of elements $a_1,\ldots,a_s\in N_A(H)$ whose images in $N_A(H)/B$ are linearly independent over $\F_2$, we have $s\le r_0+r_2$.
If $J=\langle a_1,\ldots,a_s \rangle$, then $J/(J\cap B)$ is elementary abelian of order $2^s$, and consequently $|J:J^2|=2^s$.
Now, let us write $J=K\times T$, where $K$ is torsion-free of rank $r$ and $T$ is the torsion subgroup of $J$.
If $U$ is the subgroup of all elements of $T$ of order at most $2$, then we have
\[
|J:J^2| = |K:K^2| \, |T:T^2| = 2^r |U|,
\]
since $U$ is the kernel of the homomorphism sending every element of $T$ to its square.
By the definition of $0$-rank and $2$-rank, we have $r\le r_0$ and $|U|\le 2^{r_2}$.
Thus $|J:J^2|\le 2^{r_0+r_2}$, and consequently $s\le r_0+r_2$, as desired.

(ii)
We only need to prove that a non-abelian BCI-group $G$ satisfies (\ref{BNI}).
As above, we have $G=\langle g \rangle A$, where $A$ is a non-periodic abelian group of finite $2$-rank, and $g$ is an element of order at most $4$ such that $g^2\in A$ and $a^g=a^{-1}$ for all $a\in A$.
Let $x\in G$ be such that $\langle x \rangle \ntriangleleft G$.
Then $x\not\in A$, and we can write $x=ga$ for some $a\in A$.
Since $x^2=g^2$, the order of $x$ is either $2$ or $4$.
Thus
\[
|N_G(\langle x \rangle):\langle x \rangle|
=
|N_G(\langle x \rangle):C_G(x)| \, |C_G(x):\langle x \rangle|
\le
2  \, |C_G(x):\langle x \rangle|
\]
is bounded, since $G$ is a BCI-group.
\end{proof}

As already indicated in the proof of the previous theorem, according to
\cite[Theorem 4.3]{FLTT4}, the structure of non-periodic BCI-groups is as described in (i) above, but only requiring that the $2$-rank of $A$ should be finite.
Hence we need to impose the extra condition that also the $0$-rank of $A$ should be finite  for a non-periodic BCI-group to be a BNI-group.
As a consequence, groups satisfying (\ref{BNI}) are not the same as BNI-groups, that is, imposing the condition only on {\em cyclic} non-normal subgroups instead of all non-normal subgroups is really weaker in the realm of non-periodic groups.
Observe the difference with the case of locally finite groups, where being a BNI-group, being a BCI-group, and satisfying (\ref{BNI}) were all equivalent conditions.

\end{document}